\documentclass[12pt]{amsart}

\usepackage{amssymb,amsmath}
\usepackage{amssymb,latexsym}
\usepackage{amsfonts}
\usepackage{amssymb}
\usepackage{stmaryrd}
\usepackage{longtable}
\usepackage{amsmath, geometry, amssymb} 
\usepackage{chngcntr}
\counterwithin{table}{section}
\numberwithin{equation}{section}
\newtheorem{theorem}{Theorem}[section]

\newtheorem{example}{Example}[section]
\newtheorem{remark}{Remark}[section]

\newcommand{\sqr}[2]{{\vcenter{\vbox{\hrule height#2pt
                \hbox{\vrule width#2pt height#1pt \kern#1pt
                \vrule width#2pt}\hrule height#2pt}}}}

\newcommand{\beq}{\begin{equation}}
\newcommand{\eeq}{\end{equation}}
\newcommand{\beqar}{\begin{eqnarray}}
\newcommand{\eeqar}{\end{eqnarray}}
\def\beqars{\begin{eqnarray*}}
\def\eeqars{\end{eqnarray*}}

\newcommand{\smod}[1]{\hspace{-1mm} \pmod{#1}}

\newcommand{\qu}[2]{\Bigl({\frac{#1}{#2}}\Bigr) }
\newcommand{\dqu}[2]{\ds{\qu{#1}{#2}}}

\def \ds{\displaystyle}

\newcommand{\nn}{\mathbb{N}}
\newcommand{\zz}{\mathbb{Z}}
\newcommand{\qq}{\mathbb{Q}}
\newcommand{\cc}{\mathbb{C}}

\newcommand{\hh}{\mathbb{H}}

\allowdisplaybreaks

\begin{document}

\title{A family of Eta Quotients and an Extension of the Ramanujan-Mordell Theorem}

\author{Ay\c{s}e Alaca, \c{S}aban Alaca, Zafer Selcuk Aygin}

\maketitle

\markboth{AY\c{S}E ALACA, \c{S}ABAN ALACA, ZAFER SELCUK AYGIN}
{A FAMILY OF ETA QUOTIENTS AND THE RAMANUJAN-MORDELL THEOREM}

\begin{abstract}

Let $k\geq 2$ be an integer and $j$ an integer satisfying $1\leq j \leq 4k-5$. 
We define a family  $\{ C_{j,k}(z) \}_{1\leq j \leq 4k-5} $  of eta quotients, 
and prove that this family constitute a basis for the space $S_{2k} (\Gamma_0 (12))$ of cusp forms of 
weight $2k$ and level $12$.  
We then use this basis together with certain properties of modular forms at their cusps to prove an extension of the Ramanujan-Mordell formula.\\


\end{abstract}

\section{Introduction}

Let $\nn$, $\nn_0$, $\zz$, $\qq$ and $\cc$ denote the sets of positive integers, non-negative integers, integers, rational numbers and complex numbers, respectively. 
Let $N\in\nn$. Let $\Gamma_0(N)$  be the modular subgroup defined by
\beqars
\Gamma_0(N) = \left\{ \left(
\begin{array}{cc}
a & b \\
c & d
\end{array}
\right)  \mid  a,b,c,d\in \zz ,~ ad-bc = 1,~c \equiv 0 \smod {N}
\right\} .
\eeqars 
Let $k \in \zz$. We write $M_k(\Gamma_0(N))$ to denote the space of modular forms of weight $k$ for  $\Gamma_0(N)$, and 
$E_k (\Gamma_0(N))$ and $S_k(\Gamma_0(N))$ to denote the subspaces of Eisenstein forms and cusp forms 
of  $M_k(\Gamma_0(N))$, respectively. 
It is known  
that
\beqar \label{1_10}
M_k (\Gamma_0(N)) = E_k (\Gamma_0(N)) \oplus S_k(\Gamma_0(N)).
\eeqar
The Dedekind eta function $\eta (z)$ is the holomorphic function defined on the upper half plane $\hh = \{ z \in \cc \mid \mbox{\rm Im}(z) >0 \}$ 
by the product formula
\beqars
\eta (z) = e^{\pi i z/12} \prod_{n=1}^{\infty} (1-e^{2\pi inz}).
\eeqars
An eta quotient is defined to be a finite product of the form
\beqar \label{1_30}
f(z) = \prod_{\delta } \eta^{r_{\delta}} ( \delta z),
\eeqar
where $\delta$ runs through a finite set of positive integers and the exponents $r_{\delta}$ are non-zero integers.
By taking $N$ to be the least common multiple of the $\delta$'s we can write the eta quotient {\rm (\ref{1_30})}  as 
\beqar \label{1_40}
f(z) = \prod_{1\leq \delta \mid N} \eta^{r_{\delta}} ( \delta z) ,
\eeqar 
where some of the exponents $r_{\delta}$ may be $0$.
When all  the  exponents $r_{\delta}$ are nonnegative, $f(z)$ is said to be an eta product.  

As in \cite{Kohler} 
throughout the paper we use the notation $q=e(z):=e^{2\pi i z}$ with $z\in \hh$, and so $|q| < 1$ and $q^{1/24} = e(z/24)$. 
Ramanujan's theta function $\varphi (z)$ is defined by
\beqars
\varphi(z) = \sum_{n=0}^\infty q^{  n^2 }. 
\eeqars
It is known that  $\varphi (z)$ can be expressed as an eta quotient as 
\beqar \label{1_50}
\varphi(z)=\frac{\eta^5(2z)}{\eta^2(z) \eta^2(4z)}.
\eeqar
\noindent For $a_j \in \nn $, $1 \leq j \leq 4k$, we define 
\beqars
N(a_1, \ldots, a_{4k};n):=card\{(x_1,  \ldots , x_{4k}) \in \zz^{4k} \mid  n=a_1 x_1^2  + \cdots  +a_{4k} x_{4k}^2  \} .
\eeqars
Then we have
\beqar \label{1_60}
\varphi(a_1 z)  \cdots \varphi(a_{4k} z)=\sum_{n=0}^\infty N(a_1,   \ldots,  a_{4k};n) q^{ n  }.
\eeqar
The value of  $N(a_1,  \ldots,  a_{4k};n)$ is independent of the order of the $a_j$'s. 

Let $k \geq 2$ be an integer, and let $a_j\in \{1, 3\}$,  $1 \leq j \leq 4k$, with an even number of  $a_j$'s equal to $3$. 
Then we  write 
\beqars
N(a_1,   \ldots, a_{4k};n)=N(1^{4k-2i}, 3^{2i};n), 
\eeqars
where $i$ is an integer with $0\leq i \leq 2k$.
Ramanujan \cite{ramanujan} stated a formula for $N(1^{2k}, 3^0;n)$, which was proved by Mordell in \cite{mordell}, see also \cite{cooper, aaw}. 

In this paper we define a family  $\{ C_{j,k}(z) \}_{1\leq j \leq 4k-5} $  of eta quotients, 
and prove that this family constitute a basis for the space $S_{2k} (\Gamma_0 (12))$ of cusp forms of 
weight $2k$ and level $12$.  
We then use this basis together with certain properties of modular forms at their cusps to prove an extension of the Ramanujan-Mordell formula, 
that is, we give a formula for $N(1^{4k-2i},3^{2i};n)$.

For $n, k \in \nn$ we define  the sum of divisors function $\ds \sigma_{k}(n)$ by
\beqars
\sigma_k(n) = \sum_{1\leq m \mid n} m^{k}.
\eeqars
If $n \not\in \nn$ we set $\sigma_{k }(n)=0$.   
We define the  Eisenstein series $E_{2k}(z)$ by
\beqar \label{1_80}
&&\ds E_{2k} (z)  :=-\frac{B_{2k}}{4k}+\sum_{n=1}^{\infty} \sigma_{2k-1}(n)q^{ n },
\eeqar
where $B_{2k}$ are Bernoulli numbers defined by the generating function
\beqar \label{1_90}
\ds \frac{x}{e^x-1}=\sum_{n=0}^\infty\frac{B_n x^n}{n!}.
\eeqar

The cusps of $\Gamma_0(N)$ can be represented by  rational numbers $a/c$, where $a\in \zz$, $c\in \nn$,  
$c|N$ and $\gcd(a,c)=1$, see \cite[p. 320]{okamoto} and \cite[p. 103]{DiamondShurman}. 
We can choose the representatives of cusps of $\Gamma_0(12)$  as 
\beqars
1, 1/2, 1/3, 1/4, 1/6, \infty. 
\eeqars
Throughout the paper we use $\infty$ and $1/12$ interchangeably as they are equivalent cusps for $\Gamma_0(12)$. 

Let $f(z)$ be an eta quotient given by (\ref{1_40}). 
A formula for the order $ \ds v_{a/c}(f) $ of  $f(z)$   at the cusp $a/c$ (see \cite[p. 320]{okamoto} and \cite[Proposition 3.2.8]{Ligozat})
is given by
\beqar \label{1_110}
v_{a/c}(f)=\frac{N}{24\gcd(c^2,N)}\sum_{1 \leq \delta|N}\frac{\gcd(\delta,c)^2 \cdot r_\delta }{\delta} .
\eeqar

We use the following theorem to determine if a given eta quotient is in $M_k(\Gamma_0(N))$.  
See  \cite[Proposition 1, p. 284]{Lovejoy}, \cite[Corollary 2.3, p. 37]{Kohler}, 
\cite[p. 174]{GordonSinor}, \cite{Ligozat} and \cite{Kilford}. 

\begin{theorem}{\bf (Ligozat)}
Let $f(z)$ be an eta quotient given by  (\ref{1_40})  which satisfies the following conditions:\\

{\em (L1)~} $\ds \sum_{ 1\leq  \delta \mid N} \delta \cdot r_{\delta} \equiv 0 \smod {24}$,\\

{\em (L2)~} $\ds \sum_{ 1 \leq \delta \mid N} \frac{N}{\delta} \cdot r_{\delta} \equiv 0 \smod {24}$, \\

{\em (L3)~} For each $d \mid N$, $\ds \sum_{1 \leq \delta \mid N} \frac{ \gcd (d, \delta)^2 \cdot r_{\delta} }{\delta} \geq 0 $,\\

{\em (L4)~} $\ds \sqrt{\prod_{1 \leq \delta \mid N} \delta^{ r_{\delta} } }\in \qq $, \\

{\em (L5)~} $\ds k = \frac{1}{2} \sum_{1 \leq \delta \mid N} r_{\delta} $ an even integer. \\

\noindent
Then $f(z) \in M_k(\Gamma_0(N))$.
Furthermore if all inequalities in {\em (L3)} are strict then  $f(z) \in S_k(\Gamma_0(N))$.
\end{theorem} 
\noindent For $N=12$ in (L4) of Theorem 1.1, we have
\beqar \label{1_130}
\sqrt{\prod_{1 \leq \delta \mid 12} \delta^{ r_{\delta} } }=2^{r_4+r_{12}} \sqrt{2^{ r_{2} +r_6} 3^{ r_{3}+r_6+r_{12} }  }.
\eeqar
Thus the expresion in (\ref{1_130}) is a rational number if and only if 
\beqars
 r_{2} +r_6  \equiv  0 \smod{2} \mbox{ and }
 r_{3} +r_6+r_{12}  \equiv 0 \smod{2}.
\eeqars

\section{Statements of main results}

For $j, k \in \zz$ we define  an eta quotient $C_{j,k}(z)$  by 
\beqar
C_{j,k}(z)&:=&  \Big(\frac{\eta^{10}(2z) \eta^{5}(3z) \eta(4z) \eta^{2}(6z) }{\eta^{15}(z) \eta^{3}(12z)} \Big)
\Big(\frac{\eta^{2}(2z)  \eta(3z) \eta^{3}({12z})}{ \eta^{3}(z) \eta(4z) \eta^{2}(6z) }\Big)^j  
\Big(\frac{\eta^{6}(z)\eta(6z) }{\eta^{3}(2z) \eta^{2}(3z)}\Big)^{k} \nonumber \\ \label{2_10}
&=& q^j+ \sum_{n=j+1}^\infty c_{j,k}(n) q^{  n }.
\eeqar

In the following theorem we give a basis for $M_{2k} (\Gamma_0 (12))$ when $k \geq 2$.

\begin{theorem} 
Let $k\geq 2$ be an integer. 

{\em \bf (a)} 
The family $\{ C_{j,2k}(z) \}_{1\leq j \leq 4k-5} $ constitute a basis for $S_{2k} (\Gamma_0 (12))$.

{\em \bf (b)} The set of Eisenstein series
\beqars
&& \{  E_{2k} (z),~E_{2k} (2z), ~E_{2k} (3z), ~E_{2k} (4z), ~E_{2k} (6z), ~E_{2k} ({12z}) \}
\eeqars
constitute a basis for $E_{2k}(\Gamma_0(12))$. 

{\em \bf (c)} The set 
\beqars
\{ E_{2k} (\delta z) \mid  \delta =1, 2, 3, 4, 6, 12 \} \cup \{ C_{j,2k} (z) \mid 1 \leq j \leq 4k-5 \} 
\eeqars
constitute a basis for  $M_{2k}(\Gamma_0(12))$. 
\end{theorem}

For convenience we set 
\beqar \label{2_20}
\alpha_k=\frac{-4k}{(2^{2k}-1)(3^{2k}-1)B_{2k}},
\eeqar
where $B_{2k}$ are Bernoulli numbers given in (\ref{1_90}). 
Also we write $[j]f(z):= a_j$ for $\ds f(z)=\sum_{n=0}^\infty a_n q^{n}$. 
We now give an extension of the Ramanujan-Mordell Theorem.

\begin{theorem}   
Let $k\geq 2$ be an integer and $i$ an integer satisfying $0 \leq i \leq 2k$. Let $\alpha_k$ be as in {\em(\ref{2_20})}. Then 
\begin{eqnarray*}
 \varphi^{4k-2i}(z)\varphi^{2i}(3z)= \sum_{r \mid 12} b_{(r,i,k)} E_{2k}(rz)  + \sum_{1\leq j \leq 4k-5}a_{(j,i,k)} C_{j,2k}(z) , 
\end{eqnarray*}
where
\begin{eqnarray} \label{2_30}
&& b_{(1,i,k)}=(-1)^{k}(3^{2k-i}+(-1)^{i+1}) \cdot \alpha_k,\\ \label{2_40}
&& b_{(2,i,k)}=(-1)^{i+1}(1+(-1)^{i+k})(3^{2k-i}+(-1)^{i+1})\cdot \alpha_k,\\ \label{2_50}
&& b_{(3,i,k)}=(-1)^{i+k}3^{2k-i}(3^{i}+(-1)^{i+1}) \cdot \alpha_k,\\ \label{2_60}
&& b_{(4,i,k)}=(-1)^{i}2^{2k}(3^{2k-i}+(-1)^{i+1}) \cdot \alpha_k,\\ \label{2_70}
&& b_{(6,i,k)}=-(1+(-1)^{i+k})3^{2k-i}(3^{i}+(-1)^{i+1})\cdot \alpha_k,\\ \label{2_80}
&& b_{(12,i,k)}=2^{2k}3^{2k-i}(3^{i}+(-1)^{i+1})\cdot \alpha_k,
\end{eqnarray}
and for $1 \leq j \leq 4k-5$ 
\beqar \label{2_90}
&&a_{(j,i,k)}= N(1^{4k-2i},3^{2i};j)  -\sum_{r \mid 12} b_{(r,i,k)} \sigma_{2k-1}(j/r) - \sum_{1\leq l \leq j-1}a_{(l,i,k)} c_{l,2k}(j).
\eeqar
\end{theorem}

The following theorem follows immediately from Theorem 2.2. 

\begin{theorem} 
Let $k\geq 2$ be an integer and $i$ an integer with $0 \leq i \leq 2k$. Then 
\begin{eqnarray*}
N(1^{4k-2i}, 3^{2i};n)  = \sum_{r \mid 12} b_{(r,i,k)} \sigma_{2k-1}(n/r)  + \sum_{1\leq j \leq 4k-5} a_{(j,i,k)} c_{j,2k}(n) , 
\end{eqnarray*}
where $a_{(j,i,k)}$, $c_{j,2k}(n)$  and $b_{(r,i,k)}~(r=1,2,3,4,6,12)$  are given in {\em (\ref{2_90}), (\ref{2_10})}, and  {\em (\ref{2_30})--(\ref{2_80})} respectively. 
\end{theorem}

\section{Proof of Theorem 2.1}

{\bf (a)} 
By Theorem 1.1 we see that $C_{j,2k}(z)\in S_{2k}(\Gamma_0(12))$ for $1\leq j \leq 4k-5$. 
It follows from (\ref{2_10})  that $\{ C_{j,2k}(z)\}_{1\leq j \leq 4k-5}$ is a linearly independent set. 
We deduce from the formulae in \cite[Section 6.3, p. 98]{stein}  that
\beqars
&& \ds \dim(S_{2k}(\Gamma_0(12)))=4k-5.
\eeqars
Thus the set  $\{ C_{j,2k}(z)\}_{1\leq j \leq 4k-5}$ is a basis for $S_{2k}(\Gamma_0(12))$.

{\bf (b)} It follows from  \cite[Theorem 5.9]{stein} that 
\beqars
&& \{  E_{2k} (z),~E_{2k} (2z), ~E_{2k} (3z), ~E_{2k} (4z), ~E_{2k} (6z), ~E_{2k} ({12z}) \}
\eeqars
constitute a basis for $E_{2k}(\Gamma_0(12))$. 

{\bf (c)} 
Appealing to (\ref{1_10}), the assertion follows from (a) and (b).

\section{Fourier series expansions of   $\eta (rz)$ and $E_{2k}(rz)$  at certain cusps}

Let $k \geq 2$ be an integer. 
For convenience we set  $\eta_r (z)=\eta(rz)$  for $r \in \nn$. 
We also set 
\beqar \label{4_10}
A_c= \begin{bmatrix} -1       & 0  \\ 
c       & -1 \end{bmatrix} \in SL_2(\zz).
\eeqar
The Fourier series expansions of $\eta_r(z)$ for $r=1,2,3,4,6,12$ at the cusp $1/c$ 
are given by the Fourier series expansions of $\eta_r(A_c^{-1}z)$ at the cusp $\infty$. 

In  \cite[Theorem 1.7 and Proposition 2.1]{Kohler} we take 
 $L = \begin{bmatrix}
    x       & y  \\
    u       & v 
\end{bmatrix} =L_r$  as  
\beqars
 &&L_1=\begin{bmatrix}
    -1       & 1  \\
    -1       & 0 
\end{bmatrix}, 
 L_2=\begin{bmatrix}
    -2       & 1  \\
    -1       & 0 
\end{bmatrix},    
L_3=\begin{bmatrix}
    -3       & 1  \\
    -1       & 0 
\end{bmatrix},\\
&&
L_4=\begin{bmatrix}
    -4       & 1  \\
    -1       & 0 
\end{bmatrix},
L_6=\begin{bmatrix}
    -6       & 1  \\
    -1       & 0 
\end{bmatrix}, 
 L_{12}=\begin{bmatrix}
    -12       & 1  \\
    -1       & 0
\end{bmatrix}, 
\eeqars
and 
 $A = \begin{bmatrix}
    a       & b  \\
    c       & d 
\end{bmatrix} = A_1$, where  $A_1$ is given by (\ref{4_10}).  
We obtain the Fourier series expansions of $\eta_r(z)$ for $r=1,2,3,4,6,12$ at the cusp $1$  as 
\beqar \label{4_20}
&& \ds \eta_1(A_1^{-1}z)=e^{ \pi i /3} (-z-1)^{1/2} \, \sum_{n \geq 1} \dqu{12}{n} e\Big( \frac{n^2}{24} (z+1) \Big), \\
&& \ds \eta_2(A_1^{-1}z)=\frac{e^{5 \pi i /12}}{2^{1/2}}(-z-1)^{1/2} \, \sum_{n\geq 1}\dqu{12}{n} e\Big( \frac{n^2}{48} (z+1) \Big), \\
&& \ds \eta_3(A_1^{-1}z)=\frac{e^{ \pi i /2}}{3^{1/2}}(-z-1)^{1/2} \, \sum_{n\geq 1}\dqu{12}{n} e\Big( \frac{n^2}{72} (z+1) \Big), \\
&& \ds \eta_4(A_1^{-1}z)=\frac{e^{7 \pi i /12}}{2}(-z-1)^{1/2} \, \sum_{n\geq 1}\dqu{12}{n} e\Big( \frac{n^2}{96} (z+1) \Big), \\
&& \ds \eta_6(A_1^{-1}z)=\frac{e^{3 \pi i /4}}{6^{1/2}}(-z-1)^{1/2} \, \sum_{n\geq 1}\dqu{12}{n} e\Big( \frac{n^2}{144} (z+1) \Big), \\ \label{4_70}
&& \ds \eta_{12}(A_1^{-1}z)=\frac{e^{5 \pi i /4}}{12^{1/2}}(-z-1)^{1/2} \, \sum_{n\geq 1}\dqu{12}{n} e\Big( \frac{n^2}{288} (z+1) \Big).
\eeqar
From (\ref{1_50}) and (\ref{4_20})--(\ref{4_70}) we obtain the Fourier series expansions of $\varphi (z)$ 
and $\varphi (3z)$ at  the cusp $1$ as
\beqar 
& \ds \varphi(A_1^{-1} z)&\ds =  \frac{\eta_2^5 (A_1^{-1} z)}{\eta_1^2(A_1^{-1} z)\eta_4^2(A_1^{-1}z)} \nonumber \\ \label{4_80}
&& \ds =  \frac{\ds \frac{ e^{\pi i /4}}{2^{1/2}} (-z-1)^{1/2} 
\ds \Big(\sum_{n\geq 1}\dqu{12}{n} e\big( \frac{n^2}{48} (z+1) \big) \Big )^5}
{ \ds \Big(\sum_{n\geq 1}\dqu{12}{n} e\big( \frac{n^2}{24} (z+1)\big) \Big)^2 
\ds\Big( \sum_{n\geq 1}\dqu{12}{n} e\big( \frac{n^2}{96} (z+1)\big) \Big)^2}, \\
& \ds \varphi(3 A_1^{-1}  z)&\ds =  \frac{\eta_6^5 (A_1^{-1} z)}{\eta_3^2(A_1^{-1} z)\eta_{12}^2(A_1^{-1}z)} \nonumber \\  \label{4_90}
&& =\ds\frac{  \ds \frac{e^{\pi i /4}}{6^{1/2}}  (-z-1)^{1/2}
\Big( \ds \sum_{n\geq 1}\dqu{12}{n} e\big( \frac{n^2}{144} (z+1) \big) \Big)^5}
{\Big( \ds \sum_{n\geq 1}\dqu{12}{n} e\big( \frac{n^2}{72} (z+1) \big) \Big)^2 
\Big( \ds \sum_{n\geq 1}\dqu{12}{n} e\big( \frac{n^2}{288} (z+1)\big) \Big)^2} .
\eeqar
Similarly, by taking $A=A_3$ in  \cite[Theorem 1.7 and Proposition 2.1]{Kohler}
and $L$ as  
\beqars
&&L_1:=\begin{bmatrix}
    -1       & -5  \\
    -3       & -16 
\end{bmatrix},~  
L_2:=\begin{bmatrix}
    -2       & -5  \\
    -3       & -8 
\end{bmatrix},~  
L_3:=\begin{bmatrix}
    -1       & 1  \\
    -1       & 0 
\end{bmatrix},\\
&&L_4:=\begin{bmatrix}
    -4       & -5  \\
    -3       & -4
\end{bmatrix},~
L_6:=\begin{bmatrix}
    -2       & 1  \\
    -1       & 0 
\end{bmatrix},~  
L_{12}:=\begin{bmatrix}
    -4       & 1  \\
    -1       & 0 
\end{bmatrix} 
\eeqars
we obtain the Fourier series expansions of $\eta_r(z)$ for $r=1,2,3,4,6,12$ at the cusp $1/3$ as 
\beqar \label{4_100}
&& \ds \eta_1(A_3^{-1}z)=e^{5 \pi i /3}(-3z-1)^{1/2} \, \sum_{n\geq 1}\dqu{12}{n} e\Big(\frac{n^2}{24}(z-5)\Big), \\
&& \ds \eta_2(A_3^{-1}z)= \frac{e^{7 \pi i /12}}{2^{1/2}} (-3z-1)^{1/2} \, \sum_{n\geq 1}\dqu{12}{n} e\Big(\frac{n^2}{48}(z-5)\Big), \\
&& \ds \eta_3(A_3^{-1}z)=e^{ \pi i /3} (-3z-1)^{1/2} \, \sum_{n\geq 1}\dqu{12}{n}\exp\Big(\frac{n^2}{24}(3 z+1)\Big), \\
&& \ds \eta_4(A_3^{-1}z)= \frac{e^{17 \pi i /12}}{2} (-3z-1)^{1/2} \, \sum_{n\geq 1}\dqu{12}{n}e\Big(\frac{n^2}{96}(z-5)\Big), \\
&& \ds \eta_6(A_3^{-1}z)=\frac{e^{5 \pi i /12}}{2^{1/2}}(-3z-1)^{1/2} \, \sum_{n\geq 1}\dqu{12}{n}e\Big(\frac{n^2}{48}(3 z+1 )\Big), \\ \label{4_150}
&& \ds \eta_{12}(A_3^{-1}z)= \frac{e^{7 \pi i /12}}{2} (-3z-1)^{1/2} \, \sum_{n\geq 1}\dqu{12}{n}e\Big(\frac{n^2}{96}(3 z+1 )\Big).
\eeqar
From (\ref{1_50}) and (\ref{4_100})--(\ref{4_150}) we obtain the Fourier series expansions of $\varphi (z)$ 
and $\varphi (3z)$ at  the cusp $1/3$ as
\beqar
& \ds \varphi(A_3^{-1} z) &=  \frac{\eta_2^5 (A_3^{-1} z)}{\eta_1^2(A_3^{-1} z)\eta_4^2(A_3^{-1}z)} \nonumber \\ \label{4_160}
&& =\ds \frac{  \ds \frac{e^{3\pi i /4}}{2^{1/2}}  (-3z-1)^{1/2} \Big( \ds \sum_{n\geq 1}\dqu{12}{n}e\big(\frac{n^2}{48}(z-5)\big)\Big)^5}
{\Big( \ds \sum_{n\geq 1}\dqu{12}{n}e\big(\frac{n^2}{24}(z-5)\big) \Big)^2 
\Big( \ds \sum_{n\geq 1}\dqu{12}{n}e\big(\frac{n^2}{96}(z-5)\big) \Big)^2}, \\
& \ds \varphi(3 A_3^{-1}  z) &=  \frac{\eta_6^5 (A_3^{-1} z)}{\eta_3^2(A_3^{-1} z)\eta_{12}^2(A_3^{-1}z)} \nonumber  \\ \label{4_170}
&& =\ds\frac{  \ds   \frac{e^{\pi i /4}}{2^{1/2}} (-3z-1)^{1/2} \Big( \ds \sum_{n\geq 1}\dqu{12}{n}e\big(\frac{n^2}{48}(3 z+1)\big) \Big)^5}
{\Big( \ds \sum_{n\geq 1}\dqu{12}{n}e\big(\frac{n^2}{24}(3 z+1 )\big)\Big )^2
\Big ( \ds \sum_{n\geq 1}\dqu{12}{n}e\big(\frac{n^2}{96}(3 z+1 )\big) \Big)^2} .
\eeqar
Again by taking $A=A_4$ in  \cite[Theorem 1.7 and Proposition 2.1]{Kohler} and $L$ as 
\beqars
&&L_1:=\begin{bmatrix}
    -1       & 1  \\
    -4       & 3 
\end{bmatrix},~  
L_2:=\begin{bmatrix}
    -1       & 2  \\
    -2       & 3 
\end{bmatrix},~ 
 L_3:=\begin{bmatrix}
    -3       & 1  \\
    -4       & 1 
\end{bmatrix} ,\\
&&L_4:=\begin{bmatrix}
    -1       & 4  \\
    -1       & 3 
\end{bmatrix},~  
L_6:=\begin{bmatrix}
    -3       & 2  \\
    -2       & 1 
\end{bmatrix},~  
L_{12}:=\begin{bmatrix}
    -3       & 4  \\
    -1       & 1 
\end{bmatrix}
\eeqars
we obtain the Fourier series expansions of $\eta_r(z)$ for $r=1,2,3,4,6,12$ at the cusp $1/4$ as 
\beqar \label{4_180}
&& \ds \eta_1(A_4^{-1}z)=e^{13 \pi i /12}(-4z-1)^{1/2} \, \sum_{n\geq 1}\dqu{12}{n}e\Big(\frac{n^2}{24}(z+1)\Big),\\
&& \ds \eta_2(A_4^{-1}z)=e^{ \pi i /6} (-4z-1)^{1/2} \, \sum_{n\geq 1}\dqu{12}{n}e\Big(\frac{n^2}{12}( z+ 1)\Big),\\
&& \ds \eta_3(A_4^{-1}z)=\frac{e^{5 \pi i /12}}{3^{1/2}}(-4z-1)^{1/2}\,  \sum_{n\geq 1}\dqu{12}{n}e\Big(\frac{n^2}{72}(z+1)\Big),\\
&& \ds \eta_4(A_4^{-1}z)=e^{ \pi i /12} (-4z-1)^{1/2} \, \sum_{n\geq 1}\dqu{12}{n}e\Big(\frac{n^2}{6}(z+1)\Big),\\
&& \ds \eta_6(A_4^{-1}z)=\frac{e^{ \pi i /3}}{3^{1/2}}(-4z-1)^{1/2} \, \sum_{n\geq 1}\dqu{12}{n}e\Big(\frac{n^2}{36}(z+1)\Big),\\ \label{4_230}
&& \ds \eta_{12}(A_4^{-1}z)=\frac{e^{5 \pi i /12}}{3^{1/2}} (-4z-1)^{1/2} \, \sum_{n\geq 1}\dqu{12}{n}e\Big(\frac{n^2}{18}(z+1)\Big).
\eeqar
From (\ref{1_50}) and (\ref{4_180})--(\ref{4_230}) we obtain the Fourier series expansions of $\varphi (z)$ 
and $\varphi (3z)$ at  the cusp $1/4$ as
\beqar
& \ds \varphi(A_4^{-1} z) &=  \frac{\eta_2^5 (A_4^{-1} z)}{\eta_1^2(A_4^{-1} z)\eta_4^2(A_4^{-1}z)} \nonumber  \\ \label{4_240}
&& =\ds\frac{ e^{\pi i /2}  (-4z-1 )^{1/2} \Big( \ds \sum_{n\geq 1}\dqu{12}{n}e\big(\frac{n^2}{12}( z+1)\big) \Big)^5}
{\Big( \ds \sum_{n\geq 1}\dqu{12}{n}e\big(\frac{n^2}{24}(z+1)\big)\Big)^2
\Big( \ds \sum_{n\geq 1}\dqu{12}{n}e\big(\frac{n^2}{6}(z+ 1)\big)\Big)^2}, \\
& \ds \varphi(3 A_4^{-1}  z) &=  \frac{\eta_6^5 (A_4^{-1} z)}{\eta_3^2(A_4^{-1} z)\eta_{12}^2(A_4^{-1}z)} \nonumber  \\ \label{4_250}
&& =\ds\frac{  \ds\frac{1}{3} (-4z-1)^{1/2} \Big( \ds \sum_{n\geq 1}\dqu{12}{n}e\big(\frac{n^2}{36}(z+1)\big)\Big)^5}
{\Big( \ds \sum_{n\geq 1}\dqu{12}{n}e\big(\frac{n^2}{72}( z+ 1)\big) \Big)^2 
\Big( \ds \sum_{n\geq 1}\dqu{12}{n}e\big(\frac{n^2}{18}(z+1)\big)\Big)^2} .
\eeqar

If the Fourier series expansion of a modular form $f(z)$ of weight $k$ at the cusp $1/c$ is of the form
\beqar \label{4_260}
f(A_c^{-1}z)=(-cz-1)^k\sum_{n \geq 0} a_n e^{2 \pi i n z_c}, 
\eeqar
where $z_c\in \hh$ depends on $z$ and $A_c$,  
then we refer to $(-cz-1)^k a_0$ in (\ref{4_260})  as the ``first term'' of the Fourier series expansion of $f(z)$ at the cusp $1/c$.

For $k\geq 2$ an integer and $i$ an integer with $0 \leq i \leq 2k$ we give 
the first terms of the Fourier series expansions of  $\varphi^{4k-2i}(z) \varphi^{2i} (3z)$ at certain cusps in the following table. 
We deduce them (except for the cusps $1/2$ and $1/6$) from (\ref{4_80}), (\ref{4_90}), (\ref{4_160}), (\ref{4_170}), (\ref{4_240}) and (\ref{4_250}).
\begin{center}
\begin{longtable}{l | c  c  c  c  c  c}
\caption{First  terms of $\varphi^{4k-2i}(z) \varphi^{2i} (3z)$ at certain cusps}
\endhead
{\rm cusp} & $\infty$  &  $1$  &  $ 1/2$  &  $1/3$  &  $1/4$ & $1/6$ \\
\hline \\[-2mm]
{\rm term} & $ 1 $ & $\ds (-z -1)^{2k}\frac{(-1)^k}{2^{2k}3^i}$ & $ 0$ & $\ds (-3z-1)^{2k}\frac{(-1)^{i+k}}{2^{2k}} $ & $\ds (-4z-1)^{2k}\frac{ (-1)^i}{3^i} $ & $0$   
\end{longtable}
\end{center}
By (\ref{1_110}), we have 
\beqars
&&v_{1/2}(\varphi^{4k-2i}(z)\varphi^{2i}(3z))=3k-i > 0,~\\
&&v_{1/6}(\varphi^{4k-2i}(z)\varphi^{2i}(3z))=k+i > 0
\eeqars
for all $1\leq i \leq 2k$, that is, 
the first  terms of the Fourier series expansions of $\varphi^{4k-2i}(z)\varphi^{2i}(3z)$ at cusps $1/2$ and $1/6$ are  $0$. 
This completes Table 4.1.

The following theorem is an analogue of \cite[Proposition 2.1]{Kohler} for the Eisenstein series $E_{2k}(tz)$.
For convenience we set  $E_{(2k,t)}(z)=E_{2k}(tz)$ for $t \in \nn$. 

\begin{theorem} 
Let $k\geq 2$ be an integer and $t \in \nn$.  The Fourier series expansion of $E_{(2k,t)}(z)$ at the cusp $1/c\in \qq$ is 
\beqars
E_{(2k,t)}(A_c^{-1}z)=\Big(\frac{g}{t}\Big)^{2k}(-cz-1)^{2k} E_{2k}\Big(\frac{g^2}{t}z+\frac{y g}{t}\Big),
\eeqars
where $g=\gcd(t, c)$,  $y$ is some integer, and $A_c$ is given in {\em (\ref{4_10})}.
\end{theorem}
\begin{proof}
The Fourier series expansion of $E_{(2k,t)}(z)$ at the cusp $ 1/c $ is given by 
the Fourier series expansion of $E_{(2k,t)}(A_c^{-1}z)$ at the cusp $\infty$. We have
\beqars
E_{(2k,t)}(A_c^{-1}z) =E_{(2k,t)}(\frac{-z}{-cz-1}) = E_{2k}(\frac{-tz}{-cz-1}) = E_{2k}(\gamma z),
\eeqars
where $\gamma = \begin{bmatrix}
    -t       & 0  \\
    -c       & -1 
\end{bmatrix}$. 
As $\gcd(t/g,c/g)=1$, there exist  $y,v\in \zz$ such that $\ds \frac{t}{g}(-v)+\frac{c}{g}y=1$. 
Thus $L := \begin{bmatrix}
    -t/g       & y  \\
     -c/g     & v 
\end{bmatrix} \in SL_2(\zz)$. 
Then for $k \geq 2$, we have
\beqars
& E_{(2k,t)}(A_c^{-1}z) &=E_{2k}(L L^{-1} \gamma z)\\
& &=\Big(-c\Big(\frac{(-vt +cy) z +y}{t}\Big) + v \Big)^{2k} E_{2k}\Big( \frac{(-vt +cy) z +y}{t/g}\Big)\\
& &=(g/t)^{2k}\Big(c\frac{vt-cy}{g}z + \frac{vt-cy}{g} \Big)^{2k} E_{2k}\Big( \frac{g^2 z + yg}{t}\Big)\\
& &=(g/t)^{2k}(-cz - 1 )^{2k} E_{2k}\Big( \frac{g^2 }{t}z + \frac{yg}{t}\Big),
\eeqars
which completes the proof.
\end{proof}

It folows from  Theorem 4.1 and (\ref{1_80}) that  the first term of the Fourier series expansion of $E_{2k}(tz)$ at the cusp $1/c$ is 
\beqar \label{4_270}
\Big(\frac{g}{t}\Big)^{2k} (-cz-1)^{2k} \Big( \frac{-B_{2k}}{4k} \Big).
\eeqar

\section{Proofs of Theorems 2.2 and 2.3}

Let $k\geq 2$ be an integer and $i$ an integer with $0 \leq i \leq 2k$. By (\ref{1_50}) 
we have 
\beqars
&& \ds \varphi^{4k-2i}(z)\varphi^{2i}(3z)=\frac{\eta^{20k-10i}(2z)}{\eta^{8k-4i}(z)\eta^{8k-4i}(4z)}\cdot \frac{\eta^{10i}(6z)}{\eta^{4i}(3z) \eta^{4i}({12z})}.
\eeqars
By Theorem 1.1, we have $\varphi^{4k-2i}(z)\varphi^{2i}(3z) \in M_{2k}(\Gamma_0(12))$. 
By Theorem 2.1(c), we have
\begin{eqnarray} \label{5_10}
\varphi^{4k-2i}(z)\varphi^{2i}(3z) = \sum_{r \mid 12} b_{(r,i,k)} E_{2k}(rz)  + \sum_{1\leq j \leq 4k-5}a_{(j,i,k)} C_{j,2k}(z)  
\end{eqnarray}
for some constants $b_{(1,i,k)}, b_{(2,i,k)}, b_{(3,i,k)}, b_{(4,i,k)}, b_{(6,i,k)}, b_{(12,i,k)}, a_{(1,i,k)}, \ldots, a_{(4k-5,i,k)}$. 
Since  $C_{1,k}(z), \ldots, C_{4k-5,k}(z)$ are cusp forms, the first terms of their Fourier series expansions  at all cusps are $0$. 

By appealing to (\ref{4_270}) and Table 4.1 we equate the first terms of the Fourier series expansions of (\ref{5_10}) in both sides at cusps $ \infty$, $1$, $1/2$, $1/3$, $1/4$, $1/6$  
to obtain the  system of linear equations
\beqars
 && b_{(1,i,k)} +  b_{(2,i,k)} + b_{(3,i,k)} + b_{(4,i,k)} + b_{(6,i,k)}  + b_{(12,i,k)} =  \frac{-4k}{B_{2k}} ,\\
 && b_{(1,i,k)} + \frac{b_{(2,i,k)}}{2^{2k}} + \frac{b_{(3,i,k)}}{3^{2k} } + \frac{b_{(4,i,k)}}{4^{2k}} + \frac{b_{(6,i,k)}}{6^{2k}} + \frac{b_{(12,i,k)}}{12^{2k}}
=  \frac{  (-1)^k }{2^{2k} 3^i}  \cdot \frac{-4k}{B_{2k}},\\
 && b_{(1,i,k)} +  b_{(2,i,k)}+ \frac{b_{(3,i,k)}}{ 3^{2k}} + \frac{b_{(4,i,k)}}{2^{2k} } + \frac{b_{(6,i,k)}}{3^{2k}} + \frac{b_{(12,i,k)}}{6^{2k}}=  0,\\
 && b_{(1,i,k)} +  \frac{b_{(2,i,k)} }{2^{2k}} + b_{(3,i,k)} +  \frac{b_{(4,i,k)} }{ 4^{2k}} +  \frac{b_{(6,i,k)}}{2^{2k}} + \frac{ b_{(12,i,k)}}{ 4^{2k}}
=   \frac{(-1)^{i+k}}{2^{2k}} \cdot \frac{-4k}{B_{2k}},\\
 && b_{(1,i,k)} +  b_{(2,i,k)} + \frac{b_{(3,i,k)}}{3^{2k}} +  b_{(4,i,k)} +  \frac{b_{(6,i,k)} }{3^{2k}} + \frac{b_{(12,i,k)} }{3^{2k}}
=\frac{  (-1)^{i} }{3^i} \cdot \frac{-4k}{B_{2k}},\\
 && b_{(1,i,k)} +  b_{(2,i,k)} + b_{(3,i,k)} + \frac{b_{(4,i,k)} }{ 2^{2k} } + b_{(6,i,k)}  + \frac{b_{(12,i,k)} }{ 2^{2k} }=0.
\eeqars
Solving the above system of linear equations, we obtain the asserted expressions for $b_{(r,i,k)}$ for $r=1,2,3,4,6,12$  in (\ref{2_30})--(\ref{2_80}). 
By  (\ref{2_10}), we have $[j]C_{j,k}(z) =1$ for each $j$ with $1 \leq j \leq 4k-5$. 
Equating the coefficients of $q^{ j  }$ in both sides of (\ref{5_10}) we obtain 
\beqars
N(1^{4k-2i},3^{2i};j)  = \sum_{r \mid 12} b_{(r,i,k)} \sigma_{2k-1} (j/r)  + \sum_{1\leq l \leq j-1}a_{(l,i,k)} c_{l,2k}(z) + a_{(j,i,k)}.
\eeqars
We isolate $a_{(j,i,k)}$ to complete the proof of Theorem 2.2. 
Finally, Theorem 2.3 follows from (\ref{1_60}), (\ref{1_80}), (\ref{2_10}) and Theorem 2.2.

\section{Examples and Remarks}

We now illustrate Theorems  2.2 and  2.3 by some examples.

\begin{example} 
We  determine $N(1^{6}, 3^{2};n)$ for all $n \in \nn$.  We take $k=2$  and $i=1$ in Theorems  {\em 2.2} and  {\em 2.3}.  
By {\em (\ref{2_30})--(\ref{2_80})} we have 
\beqar \label{6_10}
\left\{\begin{array}{l}
b_{(1,1,2)}= 28/5,~b_{(2,1,2)}=0,~b_{(3,1,2)}=-108/5,\\[1mm]
b_{(4,1,2)}=-448/5,~b_{(6,1,2)}=0,~b_{(12,1,2)}=1728/5.
\end{array}\right.
\eeqar
We compute  $N(1^{6}, 3^{2};n)$ for $n=1, 2, 3$ as $4k-5 =3$, and obtain
\beqar \label{6_20}
N(1^{6}, 3^{2};1)=12,~N(1^{6}, 3^{2};2)=60,~N(1^{6}, 3^{2};3)=164. 
\eeqar
By {\em (\ref{2_90})}, {\em (\ref{6_10})} and {\rm(\ref{6_20})},  we obtain
\beqars
a_{(1,1,2)}=32/5,~a_{(2,1,2)}=48,~a_{(3,1,2)}=576/5. 
\eeqars
Then, by Theorem {\em 2.3}, for all $n \in \nn$, we have 
\begin{eqnarray*}
N(1^6, 3^2;n) &=&   \frac{28}{5} \sigma_{3}(n) - \frac{108}{5} \sigma_{3}(n/3) -\frac{448}{5} \sigma_{3}(n/4) +  \frac{1728}{5} \sigma_{3}(n/12) \\
&&+ \frac{32}{5} c_{1,4}(n) +  48 c_{2,4}(n) + \frac{576}{5} c_{3,4}(n), 
\end{eqnarray*} 
which agrees with the known results, see for example \cite{alacaw}. We note that the last two cofficients in the above expression are different 
from the ones in  \cite{alacaw} since we used a different basis for the space of cusp forms.
\end{example}

\begin{example} 
We determine $N(1^4, 3^8;n)$ for all $n \in \nn$.  We take $k=3$ and $i=4$ in Theorems  {\em 2.2} and  {\em 2.3}.
By {\em (\ref{2_30})--(\ref{2_80})} we have 
\beqar \label{6_30}
\left\{\begin{array}{l}
b_{(1,4,3)}= 8/91,~b_{(2,4,3)}=0,~b_{(3,4,3)}=720/91,\\
b_{(4,4,3)}=-512/91,~b_{(6,4,3)}=0,~b_{(12,4,3)}=-4608/91.
\end{array}\right.
\eeqar
We compute  $N(1^{4}, 3^{8};n)$ for $n=1, 2, 3, 4, 5, 6, 7$ as $4k-5 =7$, and obtain
\beqar \label{6_40}
\left\{\begin{array}{l}
N(1^{4}, 3^{8};1)=8,~N(1^{4}, 3^{8};2)=24,~N(1^{4}, 3^{8};3)=48,\\
N(1^{4}, 3^{8};4)=152,~N(1^{4}, 3^{8};5)=432,\\
N(1^{4}, 3^{8};6)=720,~N(1^{4}, 3^{8};7)=1344. 
\end{array}\right.
\eeqar
By  {\em (\ref{2_90})},  {\em (\ref{6_30})} and {\em (\ref{6_40})}, we obtain
\beqars
&&a_{(1,4,3)}=720/91,~a_{(2,4,3)}=14880/91,~a_{(3,4,3)}=123376/91~a_{(4,4,3)}=40640/7,\\
&&a_{(5,4,3)}=1248448/91,~a_{(6,4,3)}=1551360/91,~a_{(7,4,3)}=792576/91. 
\eeqars
Then, by Theorem {\em 2.3}, for all $n \in \nn$, we have 
\begin{eqnarray*}
N(1^4, 3^8;n) &=&   \frac{8}{91} \sigma_{5}(n) + \frac{720}{91} \sigma_{5}(n/3) -\frac{512}{91} \sigma_{5}(n/4) 
-  \frac{46080}{91} \sigma_{5}(n/12) \\
&&+ \frac{720}{91} c_{1,6}(n) +  \frac{14880}{91} c_{2,6}(n) + \frac{123376}{91} c_{3,6}(n)  
+ \frac{40640}{7} c_{4,6}(n) \\
&&+ \frac{1248448}{91} c_{5,6}(n)  + \frac{1551360}{91} c_{6,6}(n) + \frac{792576}{91} c_{7,6}(n),
\end{eqnarray*} 
which agrees with the known results, see for example \cite{ayse-1}. 
\end{example}

\begin{example}  
Ramanujan \cite{ramanujan} stated a formula for 
$N(1^{2k},3^0,n)$, which was proved by Mordell in \cite{mordell}, see also \cite{cooper, aaw}. 
By taking  $i=0$ in Theorems {\em 2.2} and {\em 2.3}, we obtain  
\begin{eqnarray*}
N(1^{4k},3^0,n) &=& \frac{4k}{(2^{2k}-1)B_{2k}} \Big( (-1)^{k+1} \sigma_{2k-1}(n) + (1+(-1)^{k}) \sigma_{2k-1}(n/2)  \\
&&\hspace{30mm} -2^{2k} \sigma_{2k-1}(n/4) \Big)  + \sum_{1\leq j \leq 4k-5} a_{(j,0,k)} c_{j,2k}(n), 
\end{eqnarray*}
where $a_{(j,0,k)}$ and $c_{j,2k}(n)$ are given by {\em (\ref{2_90})} and {\em (\ref{2_10})}, respectively.
The coefficients of $\sigma$-functions in the above formula agree with those in   \cite[Theroem 1.1]{cooper} and \cite[Theroem 4.1]{aaw}. 
Different coefficients in the cusp part are due to the choice of our basis for the space $S_{2k}\big(\Gamma_0(12)\big)$ of cusp forms.
\end{example}

\begin{remark} 
Throughout the paper we assumed that $k \geq 2$. 
For $k=1$ we have {\em dim}$\big(S_{2}\big(\Gamma_0(12)\big)\big)=0$. A  basis for 
$M_{2}\big(\Gamma_0(12)\big) = E_{2}\big(\Gamma_0(12)\big)$ is given  
in \cite{alacaaygin}, see also \cite{williams, alaca-2, alaca-3}.
\end{remark}

\begin{remark} 
Let $k\geq 2$ be an integer. Let $N\in \nn$ and $\chi$  a Dirichlet character of modulus dividing $N$. 
We write $M_k(\Gamma_0(N), \chi)$ to denote the space of modular forms of weight $k$ with multiplier system $\chi$  for  $\Gamma_0(N)$, 
and $E_k (\Gamma_0(N), \chi)$ and $S_k(\Gamma_0(N), \chi)$ to denote the subspaces of Eisenstein forms and cusp forms 
of  $M_k(\Gamma_0(N), \chi)$, respectively.
We deduce from the formulae in \cite[Section 6.3, p. 98]{stein}  that 
\beqars
&&  {\rm dim} \big( S_{2k-1}\big(\Gamma_0(12),\chi_1\big)\big)=4k-7,\\
&&  {\rm dim}\big( S_{2k-1}\big(\Gamma_0(12),\chi_2\big)\big)=4k-6,\\
&& {\rm dim}\big( S_{2k}\big(\Gamma_0(12),\chi_3\big)\big)=4k-4, 
\eeqars
where 
\beqar \label{6_50}
\chi_1(m)=\dqu{-3}{m},~ \chi_2(m)=\dqu{-4}{m},~\chi_3(m)=\dqu{12}{m}
\eeqar
are Legendre-Jacobi-Kronecker symbols.
By appealing to a more general version of Theorem {\em 1.1 (Ligozat)}, see for example {\em \cite[Theorem 1.64]{onoweb}} and \cite[Corollary 2.3, p. 37]{Kohler}, 
we deduce that the families of eta quotients 
\beqars
&& \ds \big\{ C_{j,2k-1}(z) \big\}_{1\leq j \leq 4k-7} , \\[1mm]
&& \ds \Big\{ \frac{\eta^4(z)\eta(4z)\eta(12z)}{\eta^4(2z)\eta^2(6z)} C_{j,2k-1}(z) \Big\}_{1\leq j \leq 4k-6},\\[1mm]
&& \ds \Big\{ \frac{\eta^4(z)\eta(4z)\eta(12z)}{\eta^4(2z)\eta^2(6z)}C_{j,2k}(z)\Big\}_{1 \leq j \leq 4k-4} 
\eeqars
constitute a basis for $S_{2k-1}(\Gamma_0(12),\chi_1)$, $S_{2k-1}(\Gamma_0(12),\chi_2)$,  $S_{2k}(\Gamma_0(12),\chi_3)$, respectively, 
where $\chi_1,\chi_2,\chi_3$ are given in {\em (\ref{6_50})}. 
\end{remark}

\section*{Acknowledgments} 
The authors would like to thank Professor Shaun Cooper for bringing this research problem to our attention at CNTA XIII (2014). The research of the first two authors was supported 
by Discovery Grants from the Natural Sciences and Engineering Research Council of Canada (RGPIN-418029-2013 and RGPIN-2015-05208). Zafer Selcuk Aygin's studies are supported by Turkish Ministry of Education.

\noindent
School of Mathematics and Statistics\\
Carleton University\\
Ottawa, Ontario, Canada K1S 5B6

\vspace{2mm}

\noindent
e-mail addresses : \\
aalaca@math.carleton.ca\\
salaca@math.carleton.ca\\
zaferaygin@cmail.carleton.ca

\end{document}